\documentclass[11pt,twoside]{amsart}

\usepackage[utf8]{inputenc}
\usepackage{tikz-cd}
\usepackage{amsmath}
\usepackage{amsthm}
\usepackage{amssymb}
\usepackage{latexsym}
\usepackage{amscd}
\usepackage{ifthen}
\usepackage{hyperref}
\usepackage{graphicx}
\usepackage{enumerate}

\title[Brill-Noether loci on Enriques surfaces]{Brill-Noether loci on an Enriques surface covered by a Jacobian Kummer surface}

\author[I. Macías Tarrío, C. Spiridon, A. Stoenicǎ]{I. Macías Tarrío, C. Spiridon, A. Stoenicǎ}

\address{I. Macías Tarrío: Facultat de Matem\`atiques i Inform\`atica,
Departament de Matem\`atiques i Inform\`atica,\newline Gran Via de les Corts Catalanes
585, 08007 Barcelona, SPAIN} 
\email{irene.macias@ub.edu}

\address{C. Spiridon: Simion Stoilow Institute of Mathematics of the Romanian Academy, Research group of the project "Cohomological Hall algebras of smooth surfaces and applications" - C.F. 44/14.11.2022\newline
P.O. Box 1-764, RO-014700 Bucharest, Romania, and \newline
Faculty of Mathematics and Computer Science, University of Bucharest, Romania}
\email{cspiridon@imar.ro}

\address{A. Stoenicǎ: Faculty of Mathematics and Computer Science, University of Bucharest, Romania, and Simion Stoilow Institute of Mathematics of the Romanian Academy, Research group of the project "Cohomological Hall algebras of smooth surfaces and applications" - C.F. 44/14.11.2022} \email{andrei.stoenica@my.fmi.unibuc.ro}

\date{\today}


\let\emptyset\varnothing

\usepackage{hyperref}
\hypersetup{
    colorlinks = true,
    linkcolor = blue,
    filecolor = blue,      
    urlcolor = blue,
    citecolor = blue,
    breaklinks = true
}

\theoremstyle{plain}
\newtheorem{thm}{Theorem}[section]
\newtheorem{prop}[thm]{Proposition}
\newtheorem{cor}[thm]{Corollary}
\newtheorem{lema}[thm]{Lemma}

\theoremstyle{definition}
\newtheorem{defn}[thm]{Definition}
\newtheorem{exmp}[thm]{Example}

\newtheorem{rem}[thm]{Remark}
\newtheorem{ques}[thm]{Question}

\newcommand {\ext}{\mathrm{ext}}
\newcommand {\Sing}{\mathrm{Sing}}
\newcommand{\Pic}{\operatorname{Pic}}

\newcommand{\Supp}{\operatorname{Supp}}

\newcommand{\Hilb}{\operatorname{Hilb}}

\begin{document}

\begin{abstract}
The aim of this note is to exhibit proper first Brill-Noether loci inside the moduli spaces $M_{Y,H}(2;c_1,c_2)$ of $H$-stable rank $2$ vector bundles with fixed Chern classes of a certain type on an Enriques surface $Y$ which is covered by a Jacobian Kummer surface $X$.
\end{abstract}

\maketitle

\section{Introduction}

In the last decades, classical Brill-Noether theory, concerning line bundles on projective curves, has been extended to higher rank bundles on varieties of arbitrary dimension. In \cite{art_costa_miro-roig_2010}, Costa and Mir\'{o}-Roig introduced the Brill-Noether loci $W^k_H$, whose set theoretic support is formed by those bundles in $M_H$ having at least $k$ independent sections, where $M_H$ denotes the moduli space of $H$-stable vector bundles of some rank and some fixed Chern classes on a smooth polarized algebraic variety $(X,H)$. In a recent paper, Nugent \cite{arxiv_nugent_2024} provided a new construction of the Brill-Noether loci, generalizing the one in \cite{art_costa_miro-roig_2010}.

Standard questions of generalized Brill-Noether theory naturally arose from the ones of the classical theory and concern non-emptiness, irreducibility, dimension of the irreducible components or singularities of $W^k_H$. These problems are very far from being solved for an arbitrary variety. Little is known even in the surface case, despite the progress which has been made for the projective plane $\mathbb{P}^2$ \cite{art_gould_lee_liu_2022}, \cite{arxiv_coskun_2024}, K3 surfaces \cite{art_leyenson_2012, arxiv_leyenson_2018}, Hirzebruch surfaces \cite{art_costa_miro-roig2_2010}, ruled surfaces \cite{arxiv_irene_costa_2024} or fibered surfaces \cite{art_reyes-ahumada_roa-leguizamon_torres-lopez_2022}. We also indicate \cite{coskun_huizenga_nuer_2023} for a survey regarding higher rank Brill-Noether theory on surfaces.

In this note, we analyze the case of an Enriques surface $Y$ such that its K3 cover is a Jacobian Kummer surface and provide a way of choosing the Chern classes $c_1, c_2$ such that first Brill-Noether locus $W^1_H(2;c_1,c_2)$ is non-empty and strictly contained in the moduli space $M_{Y,H}(2;c_1,c_2)$ for any polarization $H$. The reasons for considering the above setup are a good understanding of the Picard lattice of a Jacobian Kummer surface and a theorem of F. Takemoto as reinterpreted by H. Kim \cite[Section 2]{art_kim_1998} which gives a recipe to construct rank $2$ stable vector bundles on an Enriques surface by pushing forward line bundles on its K3 cover.

The outline of this note is as follows. We start by recalling in the first part of Section \ref{section_Preliminaries} the basics of higher rank Brill-Noether theory and then we discuss Jacobian Kummer surfaces. In Section \ref{section_BN} we present sufficient conditions for line bundles on a Jacobian Kummer surface in order to obtain a proper first Brill-Noether locus on the corresponding Enriques surface. This is the content of Theorem \ref{thm_main-thm}. After this, in Propositions \ref{prop_ex1} and \ref{prop_ex2}, we give general examples of line bundles satisfying the hypothesis of Theorem \ref{thm_main-thm}. Some specific examples are also discussed.

\medskip    

\textbf{Acknowledgements.} The authors would like to thank professor Marian Aprodu for suggesting the problem and helpful discussions and the referee for carefully reading the paper and valuable comments.

\begin{small}
I. Macías Tarrío was partly supported by PID2020-113674GB-I00. C. Spiridon and A. Stoenicǎ were supported by the PNRR grant CF 44/14.11.2022 \textit{Cohomological Hall algebras of smooth surfaces and applications}.
\end{small}

\section{Preliminaries} \label{section_Preliminaries}

\subsection{Generalized Brill-Noether theory} \label{subsection_general_brill-noether}

For the purpose of this note, we will place ourselves in the surface case and work over $\mathbb{C}$. Let us fix a smooth projective surface $S$ and a polarization (i.e. ample line bundle) $H$ on $S$. The pair $(S,H)$ is usually called a \textit{polarized surface}. Recall first the definition of slope stability according to which we construct the moduli space of stable bundles on $S$.

\begin{defn}
If $\mathcal{E}$ is a vector bundle on $S$, we say that $\mathcal{E}$ is stable with respect to $H$ (or $H$-stable) if for any coherent subsheaf $\mathcal{F}$ of $\mathcal{E}$ with $0 < \operatorname{rk}(\mathcal{F}) < \operatorname{rk}(\mathcal{E})$ the following inequality holds true
\begin{align*}
    \mu_H(\mathcal{F}):=\frac{c_1(\mathcal{F})\cdot H}{\operatorname{rk}(\mathcal{F})}<\frac{c_1(\mathcal{E})\cdot H}{\operatorname{rk}(\mathcal{E})}:=\mu_H(\mathcal{E})
\end{align*}
\end{defn}

We denote by $M_H := M_{S,H}(r;c_1,c_2)$ the moduli space of $H$-stable vector bundles on $S$ of rank $r$ with fixed Chern classes $c_i$, for $i = 1,2$. We recall the following (see e.g. \cite[Section 4.5]{book_huybrechts_lehn_2010}):

\begin{thm} \label{thm_mh}
   If non-empty, $M_H$ is a quasi-projective variety and each irreducible component of it has dimension at least $2rc_2 - (r-1)c_1^2 - (r^2-1)\chi(\mathcal{O}_S)$. 
\end{thm}

\begin{thm}\cite[Theorem 2.3]{art_costa_miro-roig_2010} \label{thm_existance_BN}
     Let $(S,H)$ and $M_H = M_{S,H}(r; c_1,c_2)$ as above. We assume the following further condition:

$(\ast)$ $h^2(\mathcal{E}) = 0$ for any $\mathcal{E} \in M_H$;

Then, for any $k\geq0$, there exists a determinantal variety $W_H^k(r; c_1, c_2)$, called the \textit{$k$-th Brill-Noether loci}, such that
$$\Supp(W_H^k(r; c_1, c_2)) = \{\mathcal{E}\in M_H \ | \ h^0(\mathcal{E})\geq k\}.$$
Moreover, each non-empty irreducible component of $W_H^k(r;c_1, c_2)$ has dimension at least
\begin{align*}
  \rho_H^k(r;c_1,c_2):=& \dim(M_H) - k(k-\chi(r;c_1,c_2)) = \\
   =& \dim(M_H)-k(k-(r\chi\mathcal({O}_S) - \frac{1}{2}c_1\cdot K_S + \frac{1}{2}c_1^2 - c_2)) 
\end{align*}
and
\begin{align*}
    W_H^{k+1}(r;c_1, c_2)\subseteq \Sing(W_H^k(r;c_1, c_2))
\end{align*}
whenever $W_H^{k}(r;c_1, c_2)\neq M_H$.
\end{thm}

As one might expect, the number $\rho_H^k(r;c_1,c_2)$ is called the \textit{generalized Brill-Noether number} or the \textit{expected dimension} of $W_H^k(r;c_1.c_2)$.  

The theorem above shows in particular that the Brill-Noether loci form a filtration of the moduli space $M_{S,H}(r;c_1,c_2)$:
\[
M_H =  W_H^0(r;c_1,c_2)\supseteq W_H^1(r;c_1,c_2)\supseteq \ldots \supseteq W_H^{k}(r;c_1,c_2)\supseteq\ldots
\]

As noticed in \cite[Corollary 2.3]{art_costa_miro-roig_2010}, the vanishing assumption $(\ast)$ of Theorem \ref{thm_existance_BN} holds true when $(c_1 . H) \ge (rK_S . H)$. The latter is true in particular if the first Chern class $c_1$ is effective and the canonical divisor on $S$ in numerically trivial. Thus, for the purpose of our note, we record the following:

\begin{rem} \label{rem_vanishing}
    Let $Y$ be an Enriques surface (i.e. a smooth projective surface with $q(Y) = p_a(Y) = 0$ and $2K_Y = 0$) and $c_i\in H^{2i}(Y,\mathbb{Z})$, for $i = 1, 2$, such that $(c_1.H)>0$ for any polarization $H$. Then, for any polarization $H$, the vanishing condition $(\ast)$ holds true because otherwise there would be an injective morphism $\mathcal{O}_Y \rightarrow \mathcal{E}^{\vee} \otimes K_Y$ and from here the vanishing can be proved quickly using stability and Chern classes. A particular case which will be of interest for us in this paper is when $c_1$ is effective, then $(c_1.H)>0$ for any polarization. Consequently, for any $r \ge 2$ and for any $k\ge 0$, the $k$-th Brill-Noether locus $W^k_H(r;c_1,c_2)$ exists.
\end{rem}

\subsection{Jacobian Kummer surfaces} \label{subsection_Jacobian}

In the sequel, we recall the construction of a Jacobian Kummer surface (see e.g. \cite[Section 1]{art_keum_1997}). Let $C$ be a smooth complex projective curve of genus $2$ and $\mathcal{A} = J(C)$ its Jacobian variety. Note that $\mathcal{A}$ is an abelian surface which comes with an involution $\iota:\mathcal{A}\rightarrow \mathcal{A}$ with $16$ fixed points. The quotient $\mathcal{K} = \mathcal{A}/\iota$ can be identified with the so-called Kummer quartic, which is a surface of degree $4$ in $\mathbb{P}^3$ with the maximal possible number of $16$ singular points. Each singular point is a node. Moreover, there are precisely $16$ planes in $\mathbb{P}^3$ touching $\mathcal{K}$ along a conic. Each of these conics is called a trope. An important feature of $\mathcal{K}$ is the so-called $(16_6)$-configuration, that is any node lies on exactly $6$ tropes and each trope passes through exactly $6$ nodes.

The Jacobian Kummer surface $X = Kum(\mathcal{A})$ is the minimal desingulari\-zation of $\mathcal{K}$. The exceptional curves on $X$ lying over the nodes of $\mathcal{K}$ and the proper transforms of the tropes of $\mathcal{K}$ will be also called nodes and respectively tropes of $X$. The nodes and the tropes of the Jacobian Kummer surface $X$ form two families of $16$ mutually disjoint $(-2)$-curves on $X$. They can be described more explicitly in order to have a better understanding of the Picard lattice of $X$. For references we indicate \cite[Section 4]{art_ohashi_2009}, \cite[Section 2]{art_aprodu_kim_2020}. 

Take $p_1,\ldots, p_6$ the Weierstrass points of $C$. The $2$-torsion points of $\mathcal{A}$ are $[0]$ and $[p_i-p_j]$, $1\le i < j \le 6$. Accordingly, we denote the nodes of $X$ by $E_0$ and $E_{ij}$, $1\le i < j \le 6$. The curve $C$ has also precisely $16$ theta characteristics (i.e. divisor classes $D\in \Pic(C)$ such that $2D\sim K_C$), namely $[p_i]$, $1\le i\le 6$ and $[p_i+p_j-p_6]$ with $1\le i < j \le 5$, giving rise to $16$ theta divisors on $\mathcal{A}$, whose proper transforms are the tropes of $X$. Accordingly, we denote the tropes of $X$ by $T_{i}$, $1\le i\le 6$ and $T_{ij6}$, $1\le i<j\le 5$.

The lemma below can be easily deduced from \cite[Section 1]{art_keum_1997} and \cite[Section 2]{art_aprodu_kim_2020}:  

\begin{lema} \label{lemma_intersection_nodes_tropes}
    The intersection form $(\ .\ )$ on $X$ has the following values for the nodes and the tropes of $X$:
\begin{itemize}
    \item[1)] $(E_0.T_i) = 1$ for any $1\le i\le 6$;
    \item[2)] $(E_0. T_{ij6}) = 0$ for any $1\le i<j\le 5$;
    \item[3)] $(E_{ij}.T_k) = \left\{
        \begin{array}{ll}
             1 & \text{if } k\in \{i,j\}\\
             0 & \text{otherwise}
        \end{array}
        \right.$
        
    \item[4)] $(E_{ij}. T_{kl6}) = \left\{
        \begin{array}{ll}
             1 & \text{if } \{i,j\}\subseteq\{k,l,6\} \text{ or } \{i,j\}\cap \{k,l,6\} = \emptyset\\
             0 & \text{otherwise}
        \end{array}
        \right.$
\end{itemize}
\end{lema} 

Due to a result of Keum \cite[Theorem 2]{art_keum_1990}, there exists a fixed-point-free involution $\theta$ of $X$, giving rise to an Enriques surface $Y = X/\theta$. Let us denote by $\pi:X\rightarrow Y$ the \'etale double cover of $Y$. The set of line bundles on $X$ which come from line bundles on $Y$ via pullback is described by a result of Horikawa \cite[Theorem 5.1]{art_horikawa_1978} and recast in the following form in \cite[Lemma 2.2]{art_aprodu_kim_2020}:

\begin{lema} \label{lemma_image_of_pi_star}
    The image of the map $\pi^*:\Pic(Y) \rightarrow \Pic(X)$ is the set of line bundles $M\in \Pic(X)$ such that $\theta^*M\sim M$.
\end{lema}

The action of $\theta^*$ is presented in the table below \cite{art_aprodu_kim_2020}:

\begin{center}
    \begin{tabular}{|ccc|ccc|}
         \hline
         Nodes  & &  Tropes & Nodes & &  Tropes \\    
         \hline
         $E_{0}$ & $\longleftrightarrow$ & $T_{456}$ & $E_{25}$ & $\longleftrightarrow$ & $T_{246}$ \\
         $E_{12}$ & $\longleftrightarrow$ & $T_{3}$ & $E_{26}$ & $\longleftrightarrow$ & $T_{136}$ \\
         $E_{13}$ & $\longleftrightarrow$ & $T_{2}$ & $E_{34}$ & $\longleftrightarrow$ & $T_{356}$ \\
         $E_{14}$ & $\longleftrightarrow$ & $T_{156}$ & $E_{35}$ & $\longleftrightarrow$ & $T_{346}$ \\
         $E_{15}$ & $\longleftrightarrow$ & $T_{146}$ & $E_{36}$ & $\longleftrightarrow$ & $T_{126}$ \\
         $E_{16}$ & $\longleftrightarrow$ & $T_{236}$ & $E_{45}$ & $\longleftrightarrow$ & $T_{6}$ \\
         $E_{23}$ & $\longleftrightarrow$ & $T_{1}$ & $E_{46}$ & $\longleftrightarrow$ & $T_{5}$ \\
         $E_{24}$ & $\longleftrightarrow$ & $T_{256}$ & $E_{56}$ & $\longleftrightarrow$ & $T_{4}$ \\
         \hline
    \end{tabular}
\end{center}

\section{Brill-Noether loci} \label{section_BN}

\subsection{Main result}

Let $X$ be a Jacobian Kummer surface, $\theta$ its fixed-point-free involution, $Y = X/\theta$ and $\pi:X\rightarrow Y$ the double cover of $Y$. In the sequel, we will state and prove the main result of this note. But first, let us recall the following important result of Takemoto, as recast by H. Kim \cite[Section 2]{art_kim_1998}:

\begin{thm} \label{thm_stable_push_forward}
    Let $H$ be a polarization of $Y$. If a $\pi^*H$-stable bundle $\mathcal{E}$ is not isomorphic to $\pi^*\mathcal{F}$ for any bundle $\mathcal{F}$ on $Y$, then $\pi_*\mathcal{E}$ is $H$-stable.
\end{thm}

Combining Lemma \ref{lemma_image_of_pi_star} and Theorem \ref{thm_stable_push_forward} we obtain a recipe to construct stable rank $2$ bundles on the Enriques surface $Y$:

\begin{cor} \label{cor_pushforward_stable}
    If $M\in \Pic(X)$ is a line bundle on $X$ such that $M \not\sim \theta^*M$, then $\pi_*\mathcal{O}_X(M)$ is a stable rank $2$ bundle on $Y$ with respect to any polarization of $Y$.
\end{cor}

\begin{thm} \label{thm_main-thm}
    Let $\pi: X\rightarrow Y$ be the double cover map of an Enriques surface $Y$ by a Jacobian Kummer surface $X$. Let also $D$ be a nonzero effective divisor on $X$ such that:
    \begin{itemize}
        \item[(i)] $h^0(\mathcal{O}_X(D))=1$;
        \item[(ii)] $\theta^*D'\not\sim D'$ for any nonzero effective subdivisor $D'\subseteq D$;
        \item [(iii)] $(D^2) \le -4$.
    \end{itemize}
    
    Take  $\mathcal{V}=\pi_{*}\mathcal{O}_X(D)$ and denote by $c_i:=c_i(\mathcal{V})$ for $i=1,2$.
    Then for any ample line bundle $H$ on $Y$ one has
    $$\emptyset\neq W_H^1(2;c_1,c_2)\subsetneq M_H(2;c_1,c_2)$$
\end{thm}

\begin{proof}

First of all, condition $(ii)$ implies that $\theta^{*}D\not\sim D$ and therefore, using Corollary \ref{cor_pushforward_stable}, we get that $\mathcal{V}=\pi_{*}\mathcal{O}_X(D)$ is $H$-stable. We also know that $h^0(\mathcal{V})=h^0(\mathcal{O}_X(D))=1$ since $\pi$ is a finite map. Then, by definition, $[\mathcal{V}]\in W_H^1(2;c_1,c_2)$, ensuring us about the non-emptiness of this Brill-Noether locus. 

Secondly, let us prove that  $W_H^1(2;c_1,c_2)\subsetneq M_H(2;c_1,c_2)$. Suppose by way of contradiction that $W_H^1(2;c_1,c_2)=M_H(2;c_1,c_2)$ and let $M_0$ be the irreducible component of the moduli space $M_H(2;c_1,c_2)$ containing $[\mathcal{V}]$. Since $h^0(\mathcal{V})=1$, there exists an open subset $U\subseteq M_0$ such that $h^0(\mathcal{E}) = 1$ for any $[\mathcal{E}] \in U$.
Note that for any $[\mathcal{E}] \in U$, the vanishing locus $V(s)$ of the unique nonzero section $s\in H^0(\mathcal{E})$  has a divisorial part $C_{\mathcal{E}}$, which may be zero, and a zero-dimensional part $Z_{\mathcal{E}}$. Since the Picard group of $Y$ is discrete and $h^0(C_\mathcal{E}) = 1$ for any $[\mathcal{E}]\in U$, we may take an open subset $U'\subseteq U$ such that $C_{\mathcal{E}}=C$ for any $[\mathcal{E}]\in U'$.

\medskip

\textbf{Claim 1: $C=0$}.
 
\textit{Proof of Claim 1:} 
Any vector bundle $\mathcal{E}$ whose class belongs to $U'$ sits in an exact sequence of the form: 
\begin{equation*}
            0\rightarrow \mathcal{O}_Y(C) \rightarrow \mathcal{E} \rightarrow \mathcal{I}_{Z_\mathcal{E}}(c_1-C) \rightarrow 0
\end{equation*}
Note that $h^0(\mathcal{E} \otimes \mathcal{O}_Y(-C)) = 1$. As a consequence, since $h^0(\mathcal{V}) = 1$, we get that $h^0(\mathcal{V} \otimes \mathcal{O}_Y(-C)) = 1$. Now from the projection formula:
    \begin{align*}
    \begin{split}
        1 = h^0(\mathcal{V} \otimes \mathcal{O}_Y(-C)) &= h^0(\pi_{\ast}(\mathcal{O}_X(D) \otimes \pi^{\ast}\mathcal{O}_Y(-C))) \\
         &= h^0(\mathcal{O}_X(D) \otimes \pi^{\ast}\mathcal{O}_Y(-C)) \\
         &= h^0(\mathcal{O}_X(D - \pi^{\ast}C))
    \end{split}
    \end{align*}
Consequently, $\pi^{\ast}C$ is an effective subdivisor of $D$. However, $\pi^{\ast}C$ has to be invariant under the action of $\theta^{\ast}$, by Lemma \ref{lemma_image_of_pi_star}, but having in mind condition $(ii)$, this is not possible unless $\pi^{\ast}C = 0$ and therefore $C = 0$.
\hfill $\Box$

Now let $\mathcal{P}$ be the projective bundle over $\Hilb_{c_2}(Y)$, whose fiber over the point $Z\in \Hilb_{c_2}(Y)$ is the space of weak isomorphism classes (see \cite[p. 31]{book_friedman_1998}) of extensions of type:
\begin{equation}
\label{seq1}
            0\rightarrow \mathcal{O}_Y \rightarrow \mathcal{E} \rightarrow \mathcal{I}_{Z}(c_1) \rightarrow 0
\end{equation}

Since $h^0(\mathcal{E}) = 1$ for any $[\mathcal{E}]\in U'$, there is a well-defined injective map $U'\rightarrow\mathcal{P}$ sending $[\mathcal{E}]$ to the class of the extension (\ref{seq1}). In particular,  $\dim U'\leq \dim \mathcal{P}$.

\medskip

\textbf{Claim 2}: 
                      $\dim \mathcal{P} =  3c_2 - c_1^2/2 - 2$.

\textit{Proof of Claim 2:} 
Let us take a general point $[\mathcal{E}]$ in $U'$ and $Z\in \Hilb_{c_2}(Y)$ the vanishing locus of the unique section of $\mathcal{E}$. Let us also denote $\mathcal{J} = \mathcal{I}_Z(c_1)\otimes K_Y$. After twisting the exact sequence
\begin{align*}
    0 \rightarrow \mathcal{I}_Z \rightarrow \mathcal{O}_Y \rightarrow \mathcal{O}_Z \rightarrow 0
\end{align*}
by $K_Y(c_1)$ we get that
\begin{align*}
    \chi(\mathcal{J}) = \chi(K_Y(c_1)) - \chi(\mathcal{O}_Z)
\end{align*}
Consequently,
\begin{align} \label{eq_h1}
    h^1(\mathcal{J}) = h^0(\mathcal{J}) + h^2(\mathcal{J}) - \frac{c_1^2}{2} - 1 + c_2
\end{align}

Note that $h^0(\mathcal{J}) = 0$. Otherwise, denoting by $D_Y$ an effective divisor representing $c_1 + K_Y$, we would have an injective map $\mathcal{I}_{D_Y}  = \mathcal{O}_Y(-D_Y) \hookrightarrow \mathcal{I}_Z$, and thus an inclusion $\operatorname{Supp}(Z)\subseteq \operatorname{Supp}(D_Y)$, which is false, since $\mathcal{E}$ is general, hence so is $Z$ in $\Hilb_{c_2}(Y)$.

We also notice that, since $\mathcal{E}^\vee$ is $H$-stable of negative slope, by Serre duality we get $h^2(\mathcal{E}\otimes K_Y) = h^0(\mathcal{E^\vee}) = 0$. Thus, using the exact sequence 
\begin{align*}
0 \rightarrow K_Y \rightarrow \mathcal{E} \otimes K_Y \rightarrow I_Z(c_1) \otimes K_Y \rightarrow 0
\end{align*}
we obtain the vanishing of $h^2(\mathcal{J})$. Therefore, formula (\ref{eq_h1}) ensures us that:
\begin{align*}
    h^1(\mathcal{J}) = c_2 - \frac{c_1^2}{2} - 1
\end{align*}
Now, the claim follows since $\dim \mathcal{P} = \dim \Hilb_{c_2}(Y) + \ext^1(\mathcal{I}_Z(c_1), \mathcal{O}_Y)-1 = 2c_2 + h^1(\mathcal{I}_Z(c_1)\otimes K_Y) - 1$.
\hfill $\Box$

On the other hand, $\dim U' = \dim U \ge 4c_2-c_1^2-3\chi(\mathcal{O}_Y) = 4c_2-c_1^2-3$, by Theorem \ref{thm_mh}. Putting everything together we obtain:
\begin{align*}
    4c_2-c_1^2-3\le \dim U\le \dim \mathcal{P} = 3c_2- c_1^2/2 - 2 
\end{align*}
which implies that $c_2-\frac{1}{2}c_1^2\leq 1$.

Finally, by \cite[Propositions 27, 28]{book_friedman_1998} one has that $c_1^2=[\pi_{*}D]^2$ and $c_2=\frac{1}{2}([\pi_{*}D]^2-\pi_{*}(D^2))=\frac{1}{2}[\pi_{*}D]^2-\frac{1}{2}(D^2)$. Hence we obtain that $c_2-\frac{1}{2}c_1^2=-\frac{1}{2}(D^2) \ge 2$, which is a contradiction.
\end{proof}

Note that in order to show the strict inclusion of the first Brill-Noether locus $W_H^1(2;c_1,c_2)$ in the moduli space $M_{H}(2;c_1,c_2)$, we actually proved that the irreducible component of $M_H(2;c_1,c_2)$ to which $[\mathcal{V}]$ belongs is not contained in $W_H^1(2;c_1,c_2)$. This fact raises the following question:

\begin{ques}
    Is every irreducible component of $M_H(2;c_1,c_2)$ not contained in the first Brill-Noether locus?
\end{ques}

Another natural question would concern the non-emptiness of the second Brill-Noether locus, $W_H^2(2;c_1,c_2)$. Our previous result does prove the nontriviality of the first locus if we have a bundle with certain properties, but going further to the second locus would require finding a bundle with exactly two global sections and whose pushforward has the same Chern classes as the previous one. It seems possible to find some specific bundles on the Jacobian Kummer surface with exactly two global sections, but controlling the Chern classes of the pushforwards and making sure that they are equal to the first ones is a much more difficult task.

\begin{rem} \label{rem_expected-dim}
    It is worth mentioning that the expected dimension of the Brill-Noether locus constructed in Theorem \ref{thm_main-thm} is $\rho^1_H(2;c_1,c_2) = \operatorname{dim} M_H + 1 + \frac{1}{2}(D^2)$.
\end{rem}

The relevance of the last condition (iii) of Theorem \ref{thm_main-thm} is underlined in the following example. 

\begin{exmp}
Let $D$ be a nonzero effective divisor on $X$ and let $\mathcal{V} = \pi_*\mathcal{O}_X(D)$ and $c_i = c_i(\mathcal{V})$ for $i=1,2$. Then $h^0(\mathcal{E})\ge (D^2)/2 + 2$, for any $[\mathcal{E}]\in M_H(2;c_1,c_2)$. Indeed, by Remark \ref{rem_vanishing}, $h^2(\mathcal{E}) = 0$, hence $h^0(\mathcal{E})\ge \chi(\mathcal{E}) = \chi(\mathcal{V}) = (D^2)/2+2$. Therefore, if $D$ is either a node or a trope, then $(D^2) = -2$ and so $h^0(\mathcal{E})\ge 1$ for any $[\mathcal{E}]\in M_H(2;c_1,c_2)$, which shows that $W_H^1(2;c_1,c_2) = M_H(2;c_1,c_2)$ in this case. This fact can also be seen in the virtue of Theorem \ref{thm_existance_BN} and Remark \ref{rem_expected-dim}.
\end{exmp}

\subsection{Examples} \label{subsection_examples}

As stated before, the theorem does prove the non-emptiness and nontriviality of the first Brill-Noether locus for some specific values of the Chern classes. The next task should obviously be to show that the class of divisors with the properties mentioned in the hypothesis is non-empty. In this section, we construct general examples of divisors on a Jacobian Kummer surface $X$ satisfying the conditions of Theorem \ref{thm_main-thm}.

\begin{prop}
\label{prop_ex1}
    Let $D = \sum_{i=1}^n E_i$, with the $E_i$'s not necessarily distinct nodes. Then:
    \begin{itemize}
        \item[(i)] $h^0(\mathcal{O}_X(D)) = 1$;
        \item[(ii)] for any nonzero effective subdivisor $D'\subseteq D$ one has $\theta^*D'\not\sim D'$;
        \item[(iii)] $(D^2) \le -4$ if and only if $n \geq 2$.
    \end{itemize}
\end{prop}

\begin{proof}
$(i)$ We write $D_n = \sum_{i=1}^n E_i$.
We will first prove by induction on the number of terms in the sum the following claim: $h^1(\mathcal{O}_{D_n}) = 0$ for all $n \geq 1$.

For $n = 1$, consider the following exact sequence:
\begin{align*}
    0 \rightarrow \mathcal{O}_X(-E_1) \rightarrow \mathcal{O}_X \rightarrow \mathcal{O}_{E_1} \rightarrow 0
\end{align*}
Since $E_1$ is a $(-2)$-curve on a K3 surface, we know it is smooth and rational, therefore $\mathcal{O}_{E_1} \simeq \mathcal{O}_{\mathbb{P}^1}$ and $h^0(\mathcal{O}_{E_1}) = 1$, $h^1(\mathcal{O}_{E_1}) = h^2(\mathcal{O}_{E_1}) = 0$. Since $D_1 = E_1$, we obtain what we wanted.

Suppose now that $n \geq 2$ and that the previous claim is true for $n-1$. In this case, $D_n = E_n + D_{n-1}$. Consider the exact sequence:
\begin{align*}
    0 \rightarrow \mathcal{O}_{E_n}(-D_{n-1}) \rightarrow \mathcal{O}_{D_n} \rightarrow \mathcal{O}_{D_{n-1}} \rightarrow 0
\end{align*}
Again, $E_n \simeq \mathbb{P}^1$ and $\mathcal{O}_{E_n}(-D_{n-1}) \simeq \mathcal{O}_{\mathbb{P}^1}(-E_n \cdot D_{n-1})$. We have that $-E_n \cdot D_{n-1} \geq 0$ because the product of $E_n$ with any node is either $0$ (if they are distinct) or $-2$ (if they are equal), so $h^1(\mathcal{O}_{E_n}(-D_{n-1})) = 0$ and from the long exact sequence in cohomology we obtain $h^1(\mathcal{O}_{D_n}) = h^1(\mathcal{O}_{D_{n-1}})$, which is zero by the induction hypothesis.

Now, consider the following exact sequence:
\begin{align*}
    0 \rightarrow \mathcal{O}_X(-D_n) \rightarrow \mathcal{O}_X \rightarrow \mathcal{O}_{D_n} \rightarrow 0
\end{align*}
We now know that $h^1(\mathcal{O}_{D_n}) = h^2(\mathcal{O}_{D_n}) = 0$, therefore, from the associated long exact sequence we obtain the equality $h^2(\mathcal{O}_X(-D_n)) = h^2(\mathcal{O}_X)$. But from Serre duality $h^2(\mathcal{O}_X(-D_n)) = h^0(\mathcal{O}_X(D_n))$, and since $h^2(\mathcal{O}_X) = 1$, we obtain the desired result.

$(ii)$ For the second part, it is sufficient to prove it only for $D$, since any nonzero effective subdivisor would also be a sum of nodes and the same argument could be applied in the respective case. Write $D = \sum_{i\in I} a_i E_i$ with $a_i>0$ and $E_i$ distinct nodes. Take some $j\in I$ and simply note that $(D.E_{j})=-2a_{j}<0$, whereas $(\theta^*D.E_{j})\ge0$ (this is because, as seen in Lemma \ref{lemma_intersection_nodes_tropes} and the table below it, $\theta^*D$ will be a sum of tropes and the intersection product between a node and a trope is non-negative).\\
$(iii)$ This is immediate, since distinct nodes don't intersect and all the $E_i$'s are $-2$ curves.
\end{proof}

\begin{rem} \label{rem_sections_divisors}
The last proposition remains obviously true if we replace nodes by tropes. Notice that for $(i)$ the same argument works for sums of nodes and tropes, as long as there is no intersection between the nodes and the tropes.
\end{rem}

\begin{exmp}
According to Proposition \ref{prop_ex1}, some divisors which satisfy the conditions of Theorem \ref{thm_main-thm} could look like $E_0 + E_{12} + E_{13}$, or $3E_{23} + E_{14} + 2E_{56}$. The only thing that we should care about is having at least $2$ terms in our sum (counting multiplicities). The same idea applies for sums of tropes, as mentioned in the previous remark.
\end{exmp}

\begin{rem} \label{rem_divisorial-case}
A particularly interesting example is produced by the sum of only two nodes, for instance $D = E_0 + E_{12}$. In this case, by Theorem \ref{thm_main-thm} and Remark \ref{rem_expected-dim}, the first Brill-Noether locus $W^1_H(2;c_1,c_2)$ associated to $D$ will be a divisor in the moduli space $M_H(2;c_1,c_2)$. 
\end{rem}

For the other types of divisors mentioned in the Remark \ref{rem_sections_divisors}, we need to see how to choose some which fit the other conditions of Theorem \ref{thm_main-thm}. The last condition is the easiest to handle by far, since we know very well how the intersection product behaves. The second one is a bit trickier, and the following two results will tackle the problem of finding divisors which satisfy it.

\begin{prop}
Let $D = \sum_{i \in I} a_i E_i + T$ with the $E_i$'s distinct nodes, $a_i > 0$ and $T$ a trope such that $(E_i.T) = 0$ for all $i \in I$. Then $D \not\sim \theta^{\ast}D$ if and only if $D$ has one of the following properties:\\
a) for all $i$, $E_i \neq \theta^{\ast}T$;\\
b) there exists $i_0 \in I$ such that $E_{i_0} = \theta^{\ast}T$ and either  $a_{i_0} \geq 2$, or $a_{i_0} = 1$, but $I \setminus \{ i_0 \} \neq \emptyset$.
\end{prop}
\begin{proof} Let us first remark that $(D.T) = -2$.\\
$"\Leftarrow"$ If a) holds, then $(\theta^{\ast}D.T) = (D.\theta^{\ast}T) = 0$, so they can not be equivalent. If b) holds, then $(\theta^{\ast}D.T) = -2a_i$. In the first situation, we will again obtain that $(\theta^{\ast}D.T) \neq (D.T)$. In the second one, assume by contradiction that $D \sim \theta^{\ast}D$. Then $D -(T + \theta^{\ast}T) \sim \theta^{\ast}D -(T + \theta^{\ast}T)$. But the left hand side contains only nodes and the right hand side only tropes, and both are nonzero due to our hypothesis, therefore we arrive at a contradiction, so $D \not\sim \theta^{\ast}D$.\\
$"\Rightarrow"$ Suppose there exists $i_0 \in I$ such that $E_{i_0} = \theta^{\ast}T$. Also, suppose that $a_{i_0} = 1$. If $I = \{ i_0 \}$, then $D = T + \theta^{\ast}T = \theta^{\ast}D$ and we obtain a contradiction.
\end{proof}

\begin{exmp}
Let us first consider the divisor $D = E_{12} + T_{3} = \theta^{\ast}T_3 + T_3$. We know that the node and the trope don't intersect, but it is clear that $\theta^{\ast}D = D$. But if we had instead $D_1 = 2E_{12} + T_{3}$ or $D_2 = E_{12} + E_{14} + T_{3}$, the previous proposition implies that $D_1$ and $D_2$ are not $\theta^{\ast}$-invariant.
\end{exmp}

\begin{cor} \label{cor_theta_inv}
Let $D = \sum_{i \in I} a_i E_i + T$ with the $E_i$'s distinct, $a_i > 0$ and $T$ a trope such that $(E_i.T) = 0$ for all $i \in I$. Then $D$ has no $\theta^{\ast}$-invariant nonzero effective subdivisor if and only if for all $i$, $E_i \neq \theta^{\ast}T$.
\end{cor}

\begin{proof}
Let us first remark that in this case we have three types of nonzero effective subdivisors: $T$, some sum of nodes or $T$ plus some nodes. We can ignore the first two cases since those subdivisors are clearly not $\theta^{\ast}$-invariant due to how $\theta^{\ast}$ turns nodes into tropes and vice versa. Therefore, the only subdivisors which are relevant for us are are those from the last category and we will only focus on them.\\
$"\Rightarrow"$ In this case, $D$ is not $\theta^{\ast}$-invariant and we know its structure via the previous proposition. Assume by contradiction that we are in the second case, \textit{i.e.} there exists $i_0 \in I$ such that $E_{i_0} = \theta^{\ast}T$. But then $D$ would have a nonzero $\theta^{\ast}$-invariant effective subdivisor, namely $\theta^{\ast}T + T$, and we would have a contradiction. Therefore we are in the first case, which is what we wanted.\\
$"\Leftarrow"$ Due to our initial remark, let us consider a nonzero effective subdivisor of $D$ of the form $T$ plus some nodes which appear in $D$. Since for all $i$, $E_i \neq \theta^{\ast}T$, this subdivisor is not $\theta^{\ast}$-invariant due to the previous proposition. In consequence, $D$ has no $\theta^{\ast}$-invariant nonzero effective subdivisor, which is what we wanted to prove.
\end{proof}

\begin{rem} Let $D$ be as in the corollary with the property that $E_i \neq \theta^{\ast}T$ for all $i \in I$. Then, because of the Corollary \ref{cor_theta_inv} and the Remark \ref{rem_sections_divisors}, we know that $D$ satisfies the first two conditions of Theorem \ref{thm_main-thm}. For the last one, notice that $(D^2) = \sum_{i \in I} a_i^2 (E_i^2) + (T^2) = (-2) \cdot (\sum_{i \in I} a_i^2 + 1)$, so the third condition is also fulfilled.

\end{rem}

We gather all that was discussed previously in the following proposition which shows that we have a second family of divisors on the Jacobian Kummer surface satisfying the hypotheses of our main theorem:

\begin{prop} \label{prop_ex2}
    Let $D = \sum_{i \in I} a_i E_i + T$ with the $E_i$'s distinct, $a_i > 0$ and $T$ a trope such that $(E_i.T) = 0$ for all $i \in I$. Assume furthermore that for all $i$, $E_i \neq \theta^{\ast}T$ . Then:
    \begin{itemize}
        \item[(i)] $h^0(\mathcal{O}_X(D)) = 1$;
        \item[(ii)] for any nonzero effective subdivisor $D'\subseteq D$ one has $\theta^*D'\not\sim D'$;
        \item[(iii)] $(D^2) \le -4$.
    \end{itemize} 
\end{prop}

We finish by mentioning that the divisors described in Propositions \ref{prop_ex1} and \ref{prop_ex2} are not the only ones satisfying the hypothesis of Theorem \ref{thm_main-thm}, as the following specific example shows:

\begin{exmp}
Let us consider the divisor $D = E_0 + E_{12} + T_{136} + T_{146}$. As we can see from Lemma \ref{lemma_intersection_nodes_tropes}, the nodes and the tropes are mutually disjoint. Thus, as mentioned in Remark \ref{rem_sections_divisors}, we will have that $h^0(\mathcal{O}_X(D)) = 1$. Note that $D$ is not $\theta^*$-invariant, because $(E_0.D) = -2$, whereas $(E_0.\theta^{\ast}D) = (E_0.T_{456} + T_3 + E_{26} + E_{15}) = 1$ (see Lemma \ref{lemma_intersection_nodes_tropes} and the table describing the action of $\theta^*$ on the nodes and the tropes). With Propositions \ref{prop_ex1} and \ref{prop_ex2} in mind, we also see that $D$ has no $\theta^{\ast}$-invariant proper subdivisors. Since $(D^2) = -16 \le -4$, we conclude that $D$ satisfies the hypothesis of Theorem \ref{thm_main-thm}.
\end{exmp}

\end{document}